\documentclass[11pt,a4paper]{amsart}
\usepackage{amsmath,amsthm,latexsym,amsfonts,amssymb}
\usepackage{graphics,color,tikz}
\usepackage{mathptmx}
\usepackage[mathscr]{eucal}

\newtheorem{theorem}{Theorem}[section]
\newtheorem{lemma}[theorem]{Lemma}
\newtheorem*{definition}{Definition}

\numberwithin{equation}{section}

\renewcommand{\leq}{\leqslant}
\renewcommand{\geq}{\geqslant}

\DeclareMathOperator{\re}{Re}
\DeclareMathOperator{\im}{Im}

 \newcommand{\SL}{{\operatorname{SL}}}
 \newcommand{\GL}{{\operatorname{GL}}}
 
 \newcommand{\SU}{{\operatorname{SU}}}

\title{M\"obius Maps, Reflections and Lipschitz Constants}

\author[A. F. Beardon]{Alan F. Beardon} 
\address{Centre for Mathematical Sciences, University of Cambridge, Wilberforce Road, CB3 0WA Cambridge, United Kingdom}
\email{afb@dpmms.cam.ac.uk}
\author[T. Sugawa]{Toshiyuki Sugawa}
\address{Graduate School of Information Sciences, Tohoku University, Aoba-ku, Sendai 980-8579, Japan}
\email{sugawa@tohoku.ac.jp}
\author[M. Vuorinen]{Matti Vuorinen}
\address{Department of Mathematics and Statistics,
	University of Turku, 
	FI-20014 Turku,
	Finland}
\email{vuorinen@utu.fi}

\date{\today, FILE: \jobname.tex}

\begin{document}

\vspace{-1cm}

\begin{abstract} 
We introduce the chordal isometric circle of a M\"obius map, and use this to give a factorization of any M\"obius map as the composition of a chordal isometry and either a reflection, or a rotary reflection, across a circle. We then use this to find the chordal, and spherical, Lipschitz constants of a M\"obius map, and compare this with related results in the literature.
\end{abstract}

\maketitle



\section{Introduction}\label{Sec1}
\vspace{-12pt}

The group of $2\times 2$ non-singular complex matrices, with identity matrix $I$, is denoted by $\GL(2,\mathbb{C})$. The subgroup of matrices $A$ satisfying $\det (A) = 1$ is denoted by $\SL(2,\mathbb{C})$, and this acts on the extended complex plane (or Riemann sphere) $\mathbb{C}_\infty$ as the group
$\mathbf{M}$ of M\"obius maps $g_A$, where 
\begin{equation}\label{eqn1.1}
g_A(z) = \frac{az+b}{cz+d},\qquad
A = \begin{pmatrix}a&b\\c&d\end{pmatrix}, \quad ad-bc = 1.
\end{equation}
It is well known that $\mathbf{M}$ is the group of analytic automorphisms of the Riemann sphere $\mathbb{C}_\infty$ under the composition $(f,g)\mapsto fg$, where 
$(fg)(z) = f\big(g(z)\big)$, and we denote the
identity M\"obius map by $\mathbf{I}$. The map $A\mapsto g_A$ is a group homomorphism of $\SL(2,\mathbb{C})$ onto $\mathbf{M}$ with the kernel $\{\pm I\}$, so $\mathbf{M}$ is isomorphic to $\SL(2,\mathbb{C})/\{\pm I\}$. In particular, $g_I = \mathbf{I}$, $g_{AB} = g_Ag_B$ and $g_A^{-1} = g_{A^{-1}}$.

In the original version of \cite{For51} (reprinted in 1951) Ford defined the \emph{isometric circle} $C_g$ of a M\"obius map $g$ that does not fix $\infty$ to be the set $\{z\in \mathbb{C}\colon |g'(z)|=1\}$. Let $g$ be given by \eqref{eqn1.1}. Then, since $|g'(z)| = 1/|cz+d|^2$, $C_g$ is a Euclidean circle, and as $g$ preserves the Euclidean lengths of arcs on $C_g$ we see that $g(C_g)$ is a circle with the same radius as $C_g$. In fact, $g$ maps $C_g$ onto $C_{g^{-1}}$ because $g^{-1}g = \mathbf{I}$, and an application of the Chain rule to $g^{-1}g$ shows that $z \in C_{g}$ if and only if $g(z) \in C_{g^{-1}}$. If $c\neq 0$ then (by a direct calculation) we have $g = SR$, where
\begin{equation}\label{eqn1.2}
R(z) = 
\frac{1 - d(\overline{cz+d})}
{c(\overline{cz+d})}, \quad
S(z) = \frac{a-\overline{(cz+d)}}{c}.
\end{equation} 
and Ford gave a geometric description of the maps $R$ and $S$. First, by elementary geometry, $R$ is the inversion across the isometric circle $C_g$ of $g$. Next, $S$ is a Euclidean isometry so it must be either the reflection across a Euclidean line, or a glide reflection (that is, a translation followed by a reflection across an invariant line of the translation). Now $S$ is a reflection if and only if it has a fixed point in $\mathbb{C}$, so $S$ is a reflection if and only if $a+d$ is real (or, in the language of M\"obius maps, if and only if $g$ is elliptic, parabolic or hyperbolic, but not loxodromic). We summarize these remarks as follows.

\begin{theorem}\label{Thm1.1}
Suppose that $g$ is $g_A$ in \eqref{eqn1.1}, and that $c \neq 0$. Then $g = SR$, where $R$ is the reflection across the isometric circle $C_g$ of $g$, and $S$ is either a reflection (if $a+d$ is real) or a glide reflection (if $a+d$ is not real).
\end{theorem}

Note that we refer to the \emph{reflection} across a circle rather than the classical term `\emph{inversion}' across a circle. Informally, Theorem \ref{Thm1.1} says that $g$ and $R$ distort infinitesimal Euclidean distances by the same amount, and that there is no distortion on the isometric circle of $g$. Ford then used isometric circles to construct the Ford fundamental region for the action of a Kleinian group (that is, a discrete group of M\"obius maps) acting on $\mathbb{C}_\infty$, and this important idea is still being used today. For more details, see, for example, \cite{Bea83,For51,Leh64,Mar07,Mat98}.

We come now to the main point of this paper. Ford gives a special role to the point $\infty$ but, because of this, his approach conflicts with the modern emphasis on the homogeneity of a group action. In fact, as soon as we use the point $\infty$ (as Ford does), we should be working in \emph{inversive  geometry} (in which no point is special), and for this reason alone we should seek an alternative to Ford's approach. Here we present a natural modification of Ford's ideas in which $\infty$ does not have a special role.


\section{Chordal isometric circles} \label{Sec2}

Our modification of Ford's idea is based on the chordal distance $\chi$ on $\mathbb{C}_\infty$ which is given by 
\begin{equation*}
\chi (z,w) = \frac{2|z-w|}{\sqrt{1+|z|^2}\sqrt{1+|w|^2}}, \qquad
\chi (z,\infty) = \frac{2}{\sqrt{1+|z|^2}}.
\end{equation*}
It is well known that if $p$ is the stereographic (or linear) projection, from $(0,0,1)$, of the unit sphere $\mathbb{S}$ in $\mathbb{R}^3$ onto the extended complex plane $\mathbb{C}_\infty$, then $\chi(z,w)$ is the Euclidean distance between the points $p^{-1}(z)$ and $p^{-1}(w)$ in $\mathbb{R}^3$. As there is a discussion of the stereographic projection in almost every basic text on complex analysis, we omit the details here. 

We shall now discuss these ideas in the context of the metric space $(\mathbb{C}_\infty,\chi)$ rather than that of the Euclidean plane. First, a \emph{chordal circle} is a circle in the metric space $(\mathbb{C}_\infty,\chi)$, say $\chi(z,a)=r$, and as this is the image (under stereographic projection) of a circle on the unit sphere in $\mathbb{R}^3$, \emph{a chordal circle is either a Euclidean circle or a Euclidean line with $\infty$ attached}. This terminology avoids the unnecessary abuse of calling a straight line a circle. 
Next, the chordal distance leads naturally to the \emph{chordal derivative} $g^\chi (z)$ of a M\"obius map $g$, namely
\begin{equation*}
g^\chi (z) = \lim_{w\to z}\frac{\chi\big(g(w),g(z) \big)}{\chi(w,z)} =
\frac{(1+|z|^2)|g'(z)|}{1+|g(z)|^2},
\end{equation*}
and if $g$ is given by \eqref{eqn1.1}, then 
\begin{equation*}
g^\chi (z) = 
\frac{1+|z|^2}{|az+b|^2+|cz+d|^2},
\quad 
g^\chi (\infty) = \frac{1}{|a|^2+|c|^2}.
\end{equation*}

Finally, a M\"obius map is a \emph{chordal isometry} if it is an isometry of the metric space $(\mathbb{C}_\infty,\chi)$. We shall need the next result  (see, e.g. \cite[Theorems 2.5.1 and 4.2.2]{Bea83}), and we shall discuss chordal isometries again in Section \ref{Sec3}.

\begin{theorem}\label{Thm3.1}
Suppose that $g$ is given by \eqref{eqn1.1}.
Then the following are equivalent:\\
\begin{tabular}{rl}
{\rm (i)} &$g$ is a chordal isometry;\\
{\rm (ii)} &for all $z$, $g^\chi(z)=1$;\\
{\rm (iii)} &for all $z$, $|az+b|^2 + |cz+d|^2 = 1+|z|^2$;\\
{\rm (iv)} &$|a|^2 + |c|^2 = 1 = |b|^2 +|d|^2$ and $a\bar b + c\bar d = 0$;\\
{\rm (v)} &$d= \bar a$ and $b = -\bar c$;\\
{\rm (vi)} &$|a|^2+|b|^2+|c|^2+|d|^2 = 2$.
\end{tabular}
\end{theorem}

For a matrix $A = \begin{pmatrix}a&b\\c&d\end{pmatrix}$,
the quantity $\|A\|=\sqrt{|a|^2+|b|^2+|c|^2+|d|^2}$ is called the Frobenius norm
of $A\,.$ It is well known that $\|A\|^2 \ge 2 |{\rm det}(A)|$ \cite[p. 12]{Bea83}  and hence $\|A\|^2\ge 2$ for all $A\in \SL(2,\mathbb{C})\,.$
See also the proof of Theorem \ref{Thm3.1} in Section 3 as well as Section 5 for details.

We can now define what we mean by the chordal isometric circle of a M\"obius map.

\begin{definition}{\rm 
The \emph{chordal isometric circle} $C_g^\chi$ of a M\"obius map $g$ is  
$\{z\in \mathbb{C}_\infty \colon
g^\chi(z) = 1\}$; thus if $g$ is given by \eqref{eqn1.1} then
\begin{equation*}
C_g^\chi = 
\{z\in \mathbb{C}_\infty \colon
|az+b|^2+|cz+d|^2 = 1+|z|^2\}.
\end{equation*}
 }\end{definition}

Since
\begin{equation}\label{eqn2.2}
|az+b|^2+|cz+d|^2 -(1+|z|^2) =
\big(|a|^2+|c|^2-1\big)|z|^2 + 2\re(\bar ab+\bar cd)\bar z +\big(|b|^2+|d|^2-1\big), 
\end{equation} 
we see that the chordal isometric circle of $g$ in \eqref{eqn1.1} is also given by 
\begin{equation}\label{eqn2.1}
\big(1-(|a|^2+|c|^2)\big)|z|^2 - 2{\rm Re}\big[(\bar{a}b + \bar{c}d)\bar{z}\big] + \big(1-(|b|^2+|d|^2)\big) = 0.
\end{equation}

Clearly, $C_g^\chi$ is \emph{a Euclidean 
circle} when $|a|^2+|c|^2 \neq 1$, and 
\emph{a Euclidean line with $\infty$ attached} when $|a|^2+|c|^2 = 1$ and $\bar{a}b + \bar{c}d \neq 0$, and  
Theorem~\ref{Thm3.1} shows that if $g$ is not a chordal isometry then $C_g^\chi$ is one of these forms. Theorem~\ref{Thm3.1} also shows that if $g$ is a chordal isometry then $C_g^\chi = \mathbb{C}_\infty$, and this mirrors the Euclidean case in which if $g$ is a Euclidean isometry then the Ford isometric circle is all of $\mathbb{C}$. For this reason, we do not consider the isometric circle of any M\"obius map that is either one of these isometries. However, without doubt, the most important point here is that \emph{if $g$ is not a chordal isometry then the chordal isometric circle of $g$ is a chordal circle}.

The chordal derivative obviously satisfies the chain rule 
\begin{equation*}
(f\circ g)^\chi(z) = f^\chi\big(g(z)\big)g^\chi(z),
\end{equation*}
so that $g$ maps its chordal isometric circle onto the chordal isometric circle of $g^{-1}$.
We can now state our `chordal version' of Theorem \ref{Thm1.1}.

\begin{theorem}\label{Thm2.1}
Let $g$ be the M\"obius map in \eqref{eqn1.1}, and suppose that $g$ is not a chordal isometry. Then $g = S_\chi R_\chi$, where $R_\chi$ is the reflection across the chordal isometric circle of $g$, and $S_\chi$ is a chordal isometry.
\end{theorem}

From now on we shall use the following notation: for each M\"obius $g$ we write
\\ $\bullet$ \quad $g = S_FR_F$ for the Ford decomposition of $g$ (given in Theorem~\ref{Thm1.1}), where $R_F$ is the reflection across the Ford isometric circle, and $S_F$ is a Euclidean isometry;
\\ $\bullet$ \quad $g = S_\chi R_\chi$ for the chordal decomposition of $g$ (given in Theorem~\ref{Thm2.1}), where $R_\chi$ is the reflection across the chordal isometric circle, and $S_\chi$ is a chordal isometry.\\ 
This modification from isometric circles to chordal isometric circles provides both a bridge, and a clarification, of the close, but distinct, relationship, between Ford's Euclidean theory and the familiar theory of fundamental regions based on hyperbolic spaces and the Dirichlet fundamental region: for more details, see \cite{BSV}. The novelty here lies in the explicit expressions given below, and the application of Theorem~\ref{Thm2.1} (later in this paper) to give an elementary way of computing the Lipschitz constant of a M\"obius map with respect to the chordal distance. 

We will indeed prove Theorem \ref{Thm2.1} in Section \ref{Sec4}
by deducing the following explicit formulae.

\begin{theorem}\label{new_revised}
Under the same hypotheses as Theorem \ref{Thm2.1},
the mappings $S_\chi$ and $R_\chi$ are explicitly given by
\begin{equation}\label{eqn2.4}
S_\chi(z) = \frac{(b+\bar{c})\bar{z} -
(a-\bar{d})}{-(\bar{a}-d)\bar{z} -
(\bar{b}+c)},
\quad 
R_\chi(z) = 
\frac{(\bar{a}b+\bar{c}d)\bar{z} -
[1-(|b|^2+|d|^2)]}
{[1-(|a|^2+|c|^2)]\bar{z} - (a\bar{b}+c\bar{d})}.
\end{equation}
Moreover, when $|a|^2+|c|^2\ne 1,$ the chordal isometric circle 
$C_g^\chi$ is the Euclidean circle $|z-z_\chi| = r,$ where
\begin{equation}\label{eqn2.3}
z_\chi = \frac{\bar ab+\bar cd}{1-(|a|^2+|c|^2)},
\quad
r = \frac{\sqrt{|a|^2+|b|^2+|c|^2+|d|^2-2}}
{\big|1-(|a|^2+|c|^2)\big|}.
\end{equation}
When $|a|^2+|c|^2=1,$ $C_g^\chi$ is the Euclidean line given by
\begin{equation*}
\re\big((a\bar{b} + c\bar{d})z\big) = 1 - (|b|^2 + |d|^2).
\end{equation*}
\end{theorem}

We can show that $S_\chi$ is a chordal isometry as follows. Since $z \mapsto \bar{z}$ is a chordal isometry it is enough to show that $z \mapsto S(\bar{z})$ is a chordal isometry. Now $S(\bar{z}) = g_M(z)$, where
\begin{equation*}
M = 
\begin{pmatrix}
b+\bar{c} &\bar{d}-a\\
d-\bar{a} & -(\bar{b}+c)
\end{pmatrix}.
\end{equation*}
A calculation shows that $\|M\|^2 = -2\det(M)$ so $\det(M) < 0$. Now choose a positive $\lambda$ such that $i\lambda M \in {\rm SL}(2,\mathbb{C})$; that is, so that 
$-\lambda^2\det(M)=1$. Then $g_M = 
g_{i \lambda M}$, and $g_{i\lambda M}$ is a chordal isometry because $\|i\lambda M\|^2 = 2$.

We give two examples to illustrate Theorems \ref{Thm2.1} and \ref{new_revised} before giving their proofs. Note that, in these (and the later) examples, the explicit maps can be computed either by using Theorems \ref{new_revised}, or by making direct geometric arguments.

\noindent {\bf Example 1} \quad
Let 
\begin{equation*}
g(z) = \frac{z-3}{2z/3-1}.
\end{equation*}
Then the Ford isometric circle for $g$ is given by $|2z -3| = 3$, and the Ford decomposition of $g$ is $g = S_FR_F$, where
\begin{equation*}
R_F(z) = \frac{3\bar{z}}{2\bar{z} - 3}, 
\quad S_F(z) = 3-\bar{z}.
\end{equation*}
We can confirm this by showing directly that
$R_F(z)=z$ if and only if $|2z -3| = 3$. Clearly $S_F$ is the reflection across the line given by $x=3/2$. The chordal decomposition of $g$ is $g = S_\chi R_\chi$, where
\begin{equation*}
S_\chi(z) = \frac{7\bar{z}+6}{6\bar{z}-7}, \quad 
R_\chi(z) = 
\frac{33\bar{z} -81}{4\bar{z}-33}.
\end{equation*}

Next, $R_\chi(z)=z$ if and only if 
$4|z|^2 -66x + 81 = 0$, and one can easily check that this is indeed the chordal isometric circle of $g$. Next, $S_\chi$ is the reflection across the circle given by $S_\chi(z)=z$; that is, by the equation $|6z-7| = \sqrt{85}$. As $S_\chi$ is a chordal isometry, this must be a great circle, and we can confirm this by showing that
\begin{equation*}
\chi\left(\frac{7+\sqrt{85}}{6},
\frac{7-\sqrt{85}}{6}\right) = 2.
\end{equation*}
The invariant circles for $S_F$ and $R_F$ are illustrated (in dashed lines) in Figure \ref{fig1}, where the invariant circles for $S_\chi$ and $R_\chi$ are shown in solid lines.

\begin{figure} [ht]
\begin{center}
\begin{tikzpicture}[scale=0.25]
\draw [thick,->] (-2,0) -- (16,0);
\draw [thick,->] (0,-7) -- (0,7);

\draw [ultra thick,dashed] (1.5,0) circle [radius=1.5];
\draw [ultra thick,dashed] (1.5,-7)--(1.5,7);

\draw [ultra thick] (8.25,0) circle [radius=6.915];
\draw [ultra thick] (1.1667,0) circle [radius=1.5366];
\end{tikzpicture}
\caption{$g(z) = (z-3)/(2z/3-1)$ (Example1)}
\label{fig1}
\end{center}
\end{figure}
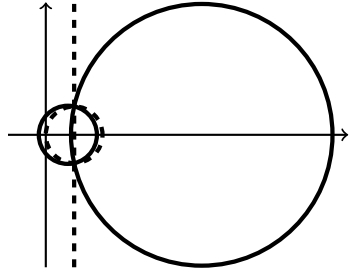

\noindent
{\bf Example 2} \quad
Let 
\begin{equation*}
g(z) = \frac{z-2}{3z-5}.
\end{equation*}
Then the Ford isometric circle for $g$ is given by $|3z-5|=1$, and $g=S_FR_F$, where
\begin{equation*}
R_F(z) = \frac{5\bar{z}-8}{3\bar{z}-5},
\quad S_F(z) = 2-\bar{z}.
\end{equation*}
Next, the chordal isometric circle for $g$ is 
$|9z - 17| = \sqrt{37}$, and $g = S_\chi R_\chi$, where
\begin{equation*}
R_\chi(z) = \frac{17\bar{z}-28}{9\bar{z}-17},
\quad S_\chi(z) = \frac{-\bar{z}+6}{6\bar{z}+1}.
\end{equation*}
A calculation shows that $S_\chi(z)$ is the reflection across the circle given by
$|6z+1| = \sqrt{37}$, and as $S_\chi$ is a chordal isometry, this circle is a great circle. The circles that are invariant under these reflections are illustrated in Figure~\ref{fig2} with the convention as before. Finally, note that the two `solid' circles are almost, but \emph{not}, tangent to each other.

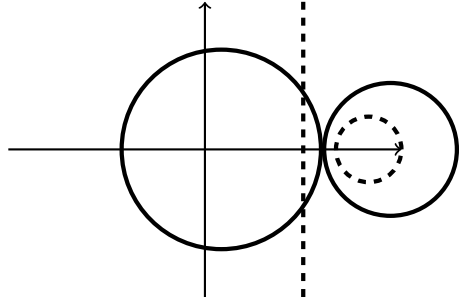
\begin{figure} [ht]
\begin{center}
\begin{tikzpicture}[scale=1.3]
\draw [thick,->] (-2,0) -- (2,0);
\draw [thick,->] (0,-1.5) -- (0,1.5);

\draw [ultra thick,dashed] (1.666,0) circle [radius=0.333];
\draw [ultra thick,dashed] (1,-1.5)--(1,1.5);

\draw [ultra thick] (1.888,0) circle [radius=0.675];
\draw [ultra thick] (0.1666,0) circle [radius=1.0138];
\end{tikzpicture}
\caption{$g(z) = (z-2)/(3z-5)$ (Example 2)}
\label{fig2}
\end{center}
\end{figure}


\section{Chordal isometries} 
\label{Sec3}

For the convenience of the reader we begin with our proof of (the known) Theorem~\ref{Thm3.1}.

\begin{proof}
[The proof of Theorem \ref{Thm3.1}]
We begin by asking the reader to verify the identity 
\begin{equation}\label{eqn3.0}
g^\chi(z)g^\chi(w) =
\left[\frac{\chi\big(g(z),g(w)\big)}{\chi(z,w)}\right]^2 = 
\frac{(1+|z|^2)(1+|w|^2)}
{\big(|az+b|^2 +|cz+d|^2\big)\,
\big(|aw+b|^2 +|cw+d|^2\big)}.
\end{equation}
The proof of the equivalence of (i)--(vi) follows easily from this identity and \eqref{eqn2.2}, and we leave the details to the reader.  To show that (iv) implies (v) we note that $(x,y) = (\bar d, - \bar b)$ is the unique solution of the equation
\begin{equation*}
\begin{pmatrix}\bar a & \bar c \\ \bar b &\bar d\end{pmatrix}
\begin{pmatrix} x \\ y \end{pmatrix}
 = \begin{pmatrix} 1 \\ 0 \end{pmatrix}.
\end{equation*}
However, if (iv) holds then $(a,c)$ is also a solution so $(a,c) = (\bar{d},-\bar{b})$, and (v) holds. Conversely, it is obvious that (v) implies (iv). Lastly, it is obvious that (iv) implies (vi). Conversely, if $A\in\SL(2,\mathbb{C})$ then $ad-bc = 1$ so that $\|A\|^2 = |a-\bar d|^2+|b+\bar c|^2 + 2$; thus (vi) implies (iv). In fact, if $A\in\SL(2,\mathbb{C})$ then
$\|A\|^2 \geq 2$ with equality if and only if $g_A$ is a chordal isometry.
\end{proof}

Theorem \ref{Thm3.1} gives algebraic information about M\"obius maps that are chordal isometries. In terms of geometry, it is well known that a M\"obius map $g$ is a chordal isometry if and only if $p^{-1}gp$ (where $p$ is the stereographic projection) is a rotation of the unit sphere $\mathbb{S}$ in $\mathbb{R}^3$. Thus if we denote the subgroup of chordal isometries in $\mathbf{M}$ by ${\rm Isom}(\chi)$, we have the next result.

\begin{lemma}\label{Lem3.2}
The group ${\rm Isom}(\chi)$ acts transitively on $\mathbb{C}_\infty$.
\end{lemma}

We have observed that the Euclidean isometry $S$ (which is now denoted by $S_F$) in Theorem \ref{Thm1.1} is either the reflection across a Euclidean line, or a glide reflection. We shall now discuss the corresponding possibilities for the chordal isometry $S$ (now denoted by $S_\chi$) in Theorem \ref{Thm2.1}.

Let $g$ be the M\"obius chordal isometry
\begin{equation}\label{eqn3.1}
g(z) = \frac{az-\bar{c}}{cz + \bar{a}}, \quad |a|^2+|c|^2 = 1,
\end{equation}
and let $S(z) = g(\bar{z})$. Since $z \mapsto \bar{z}$ is a chordal isometry, so too is $S$. Thus if $S^* = p^{-1}Sp$, then $S^*$ is an \emph{indirect} Euclidean isometry of the unit sphere $\mathbb{S}$ in $\mathbb{R}^3$ onto itself (that is, it is given by a $3\times 3$ orthogonal matrix with determinant $-1$). Now it is well known that such a map is either a \emph{reflection} $R$ of $\mathbb{R}^3$ across a plane $\Pi$ that passes through the origin $0$ of $\mathbb{R}^3$, or a \emph{rotary reflection} (that is, a reflection $R$ as just mentioned, followed by a \emph{non-trivial} rotation about the normal to $\Pi$ which passes through the origin); see, for example, \cite{Mar82},   \cite[p.168]{Neu94} and \cite[p.10]{Wil2008}. 
Now there is a simple geometric distinction between a reflection and a rotary reflection (and this is the same distinction that there is between a reflection and a glide reflection): a reflection of $\mathbb{S}$ onto itself has a fixed point on the sphere; a rotary reflection of $\mathbb{S}$ onto itself has no fixed points. Now $S^*$ has a fixed point on $\mathbb{S}$ if and only if $S$ has a fixed point in $\mathbb{C}_\infty$; that is, if and only if the equation $g(\bar{z}) = z$ has a solution in $\mathbb{C}_\infty$. Our proof of the next result is based on this observation.

\begin{theorem}\label{Thm3.3}
Let $g$ be as in \eqref{eqn3.1}. Then $S^*$ is\\
\begin{tabular}{rl}
{\rm (i)}
&a reflection if and only if \/ $\re(c) =0$;\\
{\rm (ii)}
&a rotary reflection if and only if \/ $\re(c)\neq 0$.
\end{tabular}
\end{theorem}

\begin{proof}
It is enough to prove 
\\ (i) that $\re(c) =0$ implies that $S$ has fixed points (for then $S^*$ is a reflection), and 
\\ (ii) that $\re(c)\neq 0$ implies that $S$ has no fixed points (for then $S^*$ is a rotary reflection). 
\\ Now the equation $S(z)=z$ is $g(\bar{z}) = z$, and this is
\begin{equation}\label{eqn3.2}
c|z|^2 + 2i\im(\bar{a}z) + \bar{c} = 0.
\end{equation}
We let $c = c_1 + ic_2$ and $a=a_1+ia_2.$ 
Then, by considering the real parts, we obtain $c_1\big(|z|^2 + 1) = 0$. Thus if $c_1 \neq 0$ then $S$  has no fixed points, and $S^*$ is a rotary reflection.

It remains to prove that if $c_1=0$ then the equation \eqref{eqn3.2} does have a solution in $\mathbb{C}_\infty$, so we assume now that
$c_1=0$. If $c_2= 0$ then $c=0$ and $|a|^2=1$ so that $a \neq 0$. Then \eqref{eqn3.2} reduces to 
$\im(\bar{a}z) = 0$, and this determines the Euclidean line of fixed points of $S$. Thus in this case $S^*$ does have fixed points. Now consider the remaining case; that is, when $c_1 = 0$ and $c_2\neq 0$. In this case \eqref{eqn3.2} is equivalent to the single equation
\begin{equation*}
|z|^2 + \frac{2}{c_2}\im(\bar{a}z) = 1,
\end{equation*}
and this is equivalent to the equation of a Euclidean circle, namely
\begin{equation*}
\left(x - \frac{a_2}{c_2} \right)^2 +
\left(y + \frac{a_1}{c_2} \right)^2 
=
1+ \frac{|a|^2}{c_2^2}. 
\end{equation*}
Thus in this case, $S$ has fixed points, and $S^*$ is a reflection.
\end{proof}

We give two examples. 

\noindent
{\bf Example 3} \quad
Let
\begin{equation*}
g(z) = \frac{3z+4i}{4iz+3} =
\frac{(3/5)z + (4/5)i}{(4/5)iz + (3/5)},\quad
(3/5)^2 + (4/5)^2 = 1,
\end{equation*}
and $S(z) = g(\bar{z})$.
As $\re(4i/5) = 0$ we see that $S^*$ is a 
reflection. To check this we could simply check that $g(\bar{z})=z$ has  solutions, and these solutions form the circle of fixed points of $S$ in $\mathbb{C}_\infty$.

\noindent
{\bf Example 4} \quad
Let $g(z) = -1/z$, and $S(z) = g(\bar{z})
= -1/\bar{z}$. Now it is well known that if $w \in \mathbb{C}_\infty$, then $-1/\bar{w}$ is the unique point $v$ such that $p^{-1}(w)$ and $p^{-1}(v)$ are antipodal (that is, diametrically opposite) points on $\mathbb{S}$. In other words, in this example,
$S^*(\mathbf{x}) =-\mathbf{x}$ for each $\mathbf{x}$ on $\mathbb{S}$. Clearly, $S^*$ has no fixed points on $\mathbb{S}$, so $S^*$ is a rotary reflection. In fact, $S^*$ is the reflection across $\mathbb{C}$ (the horizontal co-ordinate plane in $\mathbb{R}^3$) followed by a rotation of angle $\pi$ about the vertical co-ordinate axis.


\section{The proof of Theorem \ref{new_revised}}
\label{Sec4}

\begin{proof}[The proof of Theorem~\ref{new_revised}]
To simplify the calculations, we write
\begin{equation*}
\lambda = 1 - (|a|^2+|c|^2), \quad \mu = 1 - (|b|^2 + |d|^2), \quad \sigma = \bar{a}b + \bar{c}d.
\end{equation*}
It then follows from \eqref{eqn2.1} that the chordal isometric circle $C_g^\chi$ is given by $\lambda|z|^2 -2\re(\sigma \bar{z}) + \mu = 0$.
We first assume $\lambda \neq 0$, so 
$C_g^\chi$ is given by
\begin{equation*}
\left|z-\frac{\sigma}{\lambda}\right|^2
= \frac{|\sigma|^2-\lambda\mu}{\lambda^2}.
\end{equation*}
Since 
\begin{equation*}
\frac{\sigma}{\lambda} = 
\frac{\bar{a}b+\bar{c}d}{1-(|a|^2+|c|^2)}
\end{equation*}
this confirms the formula for $z_\chi$ in \eqref{eqn2.3}.
Next, as
\begin{align*}
|\sigma|^2 + 1
&= |\bar a b+\bar cd|^2+|ad-bc|^2 \\
&= (|a|^2+|c|^2)(|b|^2+|d|^2) \\
&= (1-\lambda)(1-\mu)\\
&= \lambda\mu - \lambda - \mu + 1,
\end{align*}
we see that 
\begin{equation*}
\frac{|\sigma|^2-\lambda\mu}{\lambda^2} =
\frac{-(\lambda+\mu)}{\lambda^2} =
\frac{|a|^2+|b|^2+|c|^2+|d|^2-2}{(1-(|a|^2+|c|^2))^2}
\end{equation*}
and this confirms the formula for $r$ in 
\eqref{eqn2.3}.

Next, we establish \eqref{eqn2.4} in the case $\lambda\ne0.$
As $C_g^\chi$ is the circle $|z-z_\chi| = r$,  the reflection $R$ across $C^\chi_g$ is given by
\begin{equation*}
R(z) = z_\chi+\frac{r^2}{\bar z-\bar z_\chi}
 = \frac{z_\chi\bar z+r^2-|z_\chi|^2}{\bar z-\bar z_\chi} =
 \frac{\sigma\bar{z} - \mu}
{\lambda\bar{z} - \bar{\sigma}},
\end{equation*}
and this gives the formula for $R_\chi$ in \eqref{eqn2.4}.

Since $R_\chi$ has order two we see that 
$S_\chi = gR_\chi$, so 
\begin{equation*}
S_\chi(z) = \frac{A\bar{z}+B}{C\bar{z}+D},
\end{equation*}
where
\begin{equation*}
\begin{pmatrix}A&B\\C&D\end{pmatrix} =
\begin{pmatrix}a&b\\c&d\end{pmatrix}
\begin{pmatrix}\sigma&-\mu \\\lambda
&-\bar{\sigma}
\end{pmatrix} = 
\begin{pmatrix} 
 a\sigma + b\lambda &-a\mu - b\bar\sigma \\
 c\sigma + d\lambda & -c\mu - d\bar\sigma
\end{pmatrix}
=
\begin{pmatrix} 
 b + \bar{c} & -a + \bar{d} \\
 -\bar{a}+d & -\bar{b}-c
\end{pmatrix}
.
\end{equation*}
We conclude that 
\begin{equation*}
S_\chi(z) = 
\frac{(a\sigma + b\lambda)\bar{z} -
(a\mu + b\bar\sigma)}
{(c\sigma + d\lambda)\bar{z} -
(c\mu + d\bar\sigma)} =
\frac{(b+\bar{c})\bar{z}-(a-\bar{d})}
{-(\bar{a}-d)\bar{z} -(\bar{b}+c)},
\end{equation*}
and this completes our proof of Theorem~\ref{new_revised} when $\lambda\ne0.$

When $\lambda=0,$ it follows from \eqref{eqn2.1} that $C_g^\chi$ is the Euclidean line given by 
$2\re(\sigma \bar z)=\sigma\bar{z}+\bar\sigma z= \mu$. 
We now see that this line is same as the set of fixed points of the map
$R(z)=(\sigma\bar{z}-\mu)/(-\bar\sigma),$ which implies that $R=R_\chi.$
Observe that this $R(z)$ agrees with the reflection $R(z)$ in the previous case.
Hence, we conclude that the same formulae are valid also when $\lambda=0.$
\end{proof}

We end this section with a result which shows when the Ford decomposition for $g$ is the same as the chordal decomposition for $g$.

\begin{theorem}\label{Thm4.1}
Suppose that a M\"obius map $g$ given by \eqref{eqn1.1} is not a chordal isometry
and that it does not fix $\infty.$
Then, in the notation above, $R_F = R_\chi$ and $S_F = S_\chi$ if and only if $d = \bar{a}$.
\end{theorem}

\begin{proof}
Suppose first that $d = \bar{a}$. Then, from \eqref{eqn1.2} and \eqref{eqn2.4}, we have 
\begin{equation*}
S_F(z) = 
-\left(\frac{\bar{c}}{c}\right)\bar{z},
\quad
S_\chi(z) = 
-\left(\frac{b+\bar{c}}{\bar{b}+c}\right)
\bar{z}.
\end{equation*}
Thus $S_\chi = S_F$ if and only if $bc = \bar{b}\bar{c}$ which is true since $bc = ad - 1 = |a|^2-1$. Thus $S_F = S_\chi$ and 
$R_F = S_Fg = S_\chi g = R_\chi$.

Now suppose that $R_F = R_\chi$ and $S_F = S_\chi$. Then the centres of the two isometric circles must coincide and, as these are $-d/c$ and $z_\chi$ (see \eqref{eqn2.3}), we have
\begin{equation*}
c\bar{a}b + c\bar{c}d = d|a|^2 + d|c|^2 - d.
\end{equation*}
Since $c\bar{c}d = |c|^2d$ and $ad-bc=1$ this reduces to $d = \bar{a}(ad-bc) = \bar{a}$.
\end{proof}

We give an example to illustrate 
Theorem~\ref{Thm4.1}.

\noindent
{\bf Example 5} \quad
Let 
\begin{equation*}
g(z) = \frac{z-1}{-2z-1} = 
\frac{(i/\sqrt{3})z -(i/\sqrt{3})}
{(-2i/\sqrt{3})z -(i/\sqrt{3})}.
\end{equation*}
Then the Ford Isometric circle for $g$ is given by $|2z+1| = \sqrt{3}$, and $g=S_FR_F$, where 
\begin{equation*}
S_F(z)=\bar{z}, \quad
R_F(z) = \frac{1-\bar{z}}{2\bar{z}+1}.
\end{equation*}
In this case, $d = \bar{a}$ so, by Theorem~\ref{Thm4.1}, we have $S_F=S_\chi$ and $R_F = R_\chi$; thus, for this $g$, the chordal decomposition of $g$ coincides with its Ford decomposition.


\section{The algebraic background}\label{Sec5}
\vspace{-12pt}

Given any $2\times 2$ complex matrix $A$ and its \emph{Hermitian transpose} $A^*$, say
\begin{equation}\label{eqn5.1}
A = \begin{pmatrix}a&b\\c&d\end{pmatrix}, 
\quad
A^* = 
\begin{pmatrix}\bar a&\bar c\\ \bar b&\bar d\end{pmatrix},
\end{equation}
we say that $A$ is \emph{unitary} if and only if $A^*A = I$; equivalently, $AA^*=I$.
As
\begin{equation}\label{eqn5.2}
A^*A = \begin{pmatrix}\bar a&\bar c\\ \bar b&\bar d\end{pmatrix}
\begin{pmatrix}a&b\\c&d\end{pmatrix}
 = \begin{pmatrix}|a|^2+|c|^2&\bar a b+\bar cd\\ a\bar b+c\bar d&|b|^2+|d|^2\end{pmatrix},
\end{equation}
we see that $A$ is unitary if and only if
\begin{equation*}
|a|^2+|c|^2 = |b|^2+|d|^2 = 1, \quad
a\bar b+c\bar d = 0.
\end{equation*}
It is now clear from Theorem \ref{Thm3.1} that a matrix $A$ in $\SL(2,\mathbb{C})$ is unitary if and only if the corresponding M\"obius map $g_A$ is a chordal isometry, and if and only if $d = \bar a$ and $b = -\bar c$. 

Next, we recall that the non-negative Frobenius norm $\|A\|$ of $A$ is defined by
\begin{equation*}
\|A\|^2 = |a|^2+|b|^2+|c|^2+|d|^2. 
\end{equation*}
Thus if $A$ is unitary then $\|A\|^2 = 2$. Also, if $A$ is unitary, then $|\det(A)|^2 = |\det(A)\det(A^*)| = 1$, so that $|\det(A)|=1$. We denote the set of unitary matrices in 
$\SL(2,\mathbb{C})$ by $\SU(2,\mathbb{C})$, and it is easy to see that this is a 
multiplicative group.

Now consider any matrix $A$ in 
$\SL(2,\mathbb{C})$. Then $\det(A^*A) = 1$ and ${\rm trace}(A^*A) = \|A\|^2$. Since a general $2\times 2$ matrix $X$ has characteristic polynomial
$t^2 -{\rm trace}(X)t + {\rm det}(X)=0$, we see that $A^*A$ has characteristic polynomial $t^2-\|A\|^2t + 1$. Thus the eigenvalues of $A^*A$ are 
\begin{equation*}
\mu_- = \tfrac12\big(\|A\|^2-\sqrt{(\|A\|^4-4}\big),
\quad\mu_+ = \tfrac12 \big(\|A\|^2+\sqrt{(\|A\|^4-4}\big),
\end{equation*}
where $0 < \mu_- \leq 1 \leq \mu_+$ and $\mu_-\mu_+ = 1$. The number
\begin{equation}\label{eqn5.3}
\mu(A) = \sqrt{\frac{\mu_+}{\mu_-}} = \mu_+
 = \frac{\|A\|^2+\sqrt{\|A\|^4-4}}{2}
\end{equation}
is known as the \emph{condition number} of the matrix $A$ (see \cite{Bj}, for example).
The condition numbers of matrices were probably first studied by A. Turing in 1948 (see \cite[p.~126]{h}), and since then condition numbers have become standard tools
in numerical linear algebra as 
they play an important role in the study of numerical stability of the solutions of systems of linear equations.
For a discussion of the Frobenius norm and the various condition numbers of matrices, see
\cite[pp. 14-24]{BMR}.

Now it is well known (from the theory of Hermitian matrices) that, because $0 < \mu_- \leq \mu_+$, there is a matrix $W$ in $\SU(2,\mathbb{C})$ that diagonalizes $A^*A$; explicitly,  
\begin{equation}\label{eqn5.4}
W^*(A^*A)W = 
\begin{pmatrix}
\mu_+&0 \\ 0&\mu_-
\end{pmatrix} = D^2,
\end{equation}
say, where $D$ is the diagonal matrix with diagonal entries $\sqrt{\mu_-}$ and $\sqrt{\mu_+}$. This observation provides us with a proof of the next lemma in which the formula $A = UDV$ is known as the \emph{singular value decomposition} of $A$.

\begin{lemma}\label{L4.1}
Suppose that $A$ in \eqref{eqn5.1} is in $ \in \SL(2,\mathbb{C})$. Then there are matrices $U$ and $V$ in $\SU(2,\mathbb{C})$ such that 
\begin{equation}\label{eqn5.5}
A = U
\begin{pmatrix}
\sqrt{\mu_+} &0 \\ 0&\sqrt{\mu_-}
\end{pmatrix}
V.\end{equation} 
\end{lemma}

\begin{proof}
Let $U = AWD^{-1}$ and $V = W^*$. Then certainly, $A = UDW^* = UDV$. Also, since $W$ is in $\SU(2,\mathbb{C})$, so too is $V$. Finally, a straightforward argument (which we leave to the reader) shows that $U^*U=I$ so that $U \in \SU(2,\mathbb{C})$.
\end{proof}


\section{Lipschitz constants of M\"obius maps}\label{Sec6}
\vspace{-12pt}

In this section we consider the Lipschitz constants of a M\"obius map $g$ with respect to the chordal distance $\chi $ and the spherical, or great-circle, distance $\bar\chi$ on $\mathbb{C}_\infty$. To illustrate these ideas we begin with two known results of this type:
\\ (1) \quad
the Euclidean Lipschitz constant of the M\"obius automorphism
$z \mapsto (z-a)/(1-\bar a z)$ on the unit disk is $(1+|a|)/(1-|a|)$ 
(\cite[p.40 Thm 5.5.1]{Bea83});
\\ (2) \quad 
the chordal Lipschitz constant of the inversion $z \mapsto p+ {r^2}/(\overline{z}-\overline{p})$, where $r^2< 1+ |p|^2$, is 
\begin{equation*}
\left(\frac{u+1+ |p|^2-r^2}{2r} \right)^2,
 \quad u = \sqrt{(1+(|p|+r)^2)(1+(|p|-r)^2)}.
\end{equation*}
(\cite[p.43, Lemma 3.29]{Har20}).

We note that (2) can be put in an invariant  form, namely that if $R$ is the reflection across a chordal circle $C$ in $\mathbb{C}_\infty$, then 
\begin{equation*}
{\rm Lip}_\chi(R) = \frac{1+\sqrt{1-\delta^2}}{1-\sqrt{1-\delta^2}},
\end{equation*}
where $\delta = \tfrac12 {\rm diam}_\chi(C)$.
To see this observe that if $h$ is a chordal isometry then $hRh^{-1}$ is the reflection across $h(C)$ and, obviously, ${\rm Lip}_\chi(hRh^{-1}) = {\rm Lip}_\chi(R)$. In this way we may assume that
$C$ is given by $|z|=r$, where $0<r\leq 1$. It now follows from (2) (with $p=0$) that 
${\rm Lip}_\chi(R) = 1/r^2$. Now $\delta = \frac12\chi(-r,r)$ and if we put $r = 1/\exp(t)$, where $t \geq 0$, we find that $\delta = 1/\cosh t$. Then $\sqrt{1-\delta^2} = \tanh t$, so
\begin{equation*}
\frac{1+\sqrt{1-\delta^2}}{1-\sqrt{1-\delta^2}} =
\frac{1+\tanh t}{1-\tanh t} = 
\frac{\cosh t + \sinh t}{\cosh t - \sinh t} = \frac{1}{\exp(2t)} = r^2 = {\rm Lip}_\chi(R).
\end{equation*}

We now begin our general discussion of the Lipschitz constants of M\"obius maps. The infinitesimal form of the chordal distance gives rise to the Riemannian metric
\begin{equation*}
\lim_{w\to z}\frac{\chi (z,w)}{|z-w|}|dz| = \frac{2|dz|}{1+|z|^2},
\end{equation*}
which is called the spherical, or Fubini-Study, metric on $\mathbb{C}_\infty$, and this induces the spherical, or great-circle, distance
\begin{equation*}
\bar\chi(z,w) = \inf_\gamma \int_\gamma\frac{2|dz|}{1+|z|^2} 
\end{equation*}
on $\mathbb{C}_\infty$, where the infimum is taken over all rectifiable curves joining $z$ and $w$ in $\mathbb{C}_\infty$. By considering the angle $\theta$ (in $[0,\pi]$) between the two line segments $[0,p^{-1}(z)]$ and $[0,p^{-1}(w)]$ in $\mathbb{R}^3$, we see that
\begin{equation*}
\tfrac12 \chi(z,w) = \sin \tfrac12 \bar\chi(z,w);
\end{equation*}
thus a M\"obius map $g$ is a chordal isometry if and only if it is a spherical isometry.

We shall now consider the two Lipschitz constants
\begin{equation*}
{\rm Lip}_\chi(g) = \sup_{z,w\in\mathbb{C}_\infty, z \neq w} \frac{\chi (g(z),g(w))}{\chi (z,w)},
\qquad
{\rm Lip}_{\bar\chi}(g) = \sup_{z,w\in\mathbb{C}_\infty, z \neq w} \frac{\bar\chi(g(z),g(w))}{\bar\chi(z,w)},
\end{equation*}
and the quantity
\begin{equation*}
\|g^\chi \|_\infty = \sup_{z\in\mathbb{C}_\infty} g^\chi (z),
\end{equation*}
and discuss (in some detail) the following known result (see \cite[Theorem 2.4]{BL}).

\begin{theorem}\label{Thm6.1}
Suppose that $A\in\SL(2,\mathbb{C})$, and let $g = g_A$. Then
\begin{equation}\label{eqn6.1}
{\rm Lip}_\chi(g) = {\rm Lip}_{\bar\chi}(g) = \|g^\chi \|_\infty = \mu(A) = \frac{\|A\|^2+\sqrt{\|A\|^4-4}}{2}.
\end{equation}
\end{theorem}

Some of these relations are already recorded in the literature, in different circumstances (for example, see \cite[p.33]{Bea91}) and with different proofs, and as there does not seem to be a coherent account anywhere, it may be helpful to give a summary, and a comparison, of these results and their old, and new, proofs. 

The last equation in \eqref{eqn6.1} is \eqref{eqn5.3}, so we need only consider the first four terms in \eqref{eqn6.1}. Next, \eqref{eqn6.1} certainly holds when $g$ is a chordal isometry since then all terms are equal to $1$. Thus, from now on, we may assume that $g$ is not a chordal isometry, and the next lemma provides the key fact in this case.

\begin{lemma}\label{Lem6.2}
Suppose that the M\"obius map $g$ is not a chordal isometry. Then there is some $\mu$ with $\mu > 1$, and chordal isometries $h_1, h_2$ such that $h_1gh_2(z) = \mu z$.
\end{lemma}

Lemma \ref{Lem6.2} is Theorem 1 in \cite{BM}, where it is proved using geometric ideas and the double coset decomposition of a group. If $G$ is any group, and $H$ is a subgroup of $G$, then the double cosets are the equivalence classes $HgH$. For example, the singular value decomposition of a matrix $A$ described in Lemma~\ref{L4.1} is an example of the use of double cosets for the group
$\SL(2,\mathbb{C})$ with respect to the subgroup $\SU(2,\mathbb{C})$. In fact, it is easy to see that $\mu$ in Lemma \ref{Lem6.2} is the condition number of the matrix $g$.

It is clear that the three values ${\rm Lip}_\chi(g)$, ${\rm Lip}_{\bar\chi}(g)$ and 
$\|g^\chi\|_\infty$ are unchanged if we replace $g$  by $hg$, or $gh$, where $h$ is any chordal isometry. It is also known (and easy to check) that if $A \in \SL(2,\mathbb{C})$  and $X \in \SU(2,\mathbb{C})$, then $\|XA\| = \|A\| =\|AX\|$; see \cite[Section 2.5]{Bea83}. These facts combined with Lemma \ref{Lem6.2} show that \eqref{eqn6.1} holds for every M\"obius map $g$ \emph{if and only if it holds for all $g$ of the form} $g(z) = \mu z$. Now it is straightforward to check that $\eqref{eqn6.1}$ holds for all maps of the form $z\mapsto \mu z$ (we leave this as an exercise for the reader), and this completes our first proof of Theorem \ref{Thm6.1}.

We turn now to other ideas and other proofs. First, an alternative proof of ${\rm Lip}_\chi(g) = \|g^\chi \|_\infty$ can be found in 
\cite[Lemma 2.1]{BL}, and this is as follows. We know from \eqref{eqn3.0}
that for a M\"obius map $g$,
\begin{equation*}
\frac{\chi (g(z),g(w))}{\chi (z,w)} = \sqrt{g^\chi (z)g^\chi (w)},
\end{equation*} 
and ${\rm Lip}_\chi(g) = \|g^\chi \|_\infty$  follows immediately from this. Moreover, it is also shown in \cite{BL} that $g^\chi$ attains its maximum value at exactly one point in $\mathbb{C}_\infty$.

Next, it is obvious that $\chi \leq \bar\chi$ and 
\begin{equation*}
g^\chi(z) = 
\lim_{w\to z} \frac{\chi (g(z),g(w))}{\chi (z,w)}
=
\lim_{w \to z} \frac{\bar\chi(g(z),g(w))}{\bar\chi(z,w)}
\end{equation*}
so that 
\begin{equation*}
\|g^\chi \|_\infty\leq {\rm Lip}_\chi(g) 
\leq {\rm Lip}_{\bar\chi}(g).
\end{equation*} 

On the other hand, for any smooth curve $\gamma$ 
joining $z$ and $w$ in $\mathbb{C}_\infty$, we have
\begin{align*}
\bar\chi(g(z),g(w))
&\leq \int_{g(\gamma)} 
\frac{2\,|ds|}{1+|s|^2}\\[4pt]
&= \int_\gamma 
\frac{2|g'(t)|\,|dt|}{1+|g(t)|^2} \\[4pt]
&= \int_\gamma\frac{2g^\chi(t)\,|dt|}{1+|t|^2}
\\[4pt]
&\leq \|g^\chi\|_\infty\,
\int_\gamma\frac{2\,|dt|}{1+|t|^2}.
\end{align*}
If we now take the infimum of the last term over all possible curves $\gamma$ we obtain
$\bar\chi(g(z),g(w))\leq \|g^\chi \|_\infty \bar\chi(z,w)$. Since $z$ and $w$ are arbitrary, we conclude that ${\rm Lip}_{\bar\chi}(g)\le\|g^\chi \|_\infty$, and this proves that
for each M\"obius map $g$ we have
\begin{equation*}
{\rm Lip}_\chi(g) = {\rm Lip}_{\bar\chi}(g) = \|g^\chi \|_\infty.
\end{equation*}

Finally, there is yet another approach which uses hyperbolic geometry on the upper half $\mathbb{H}^3$ of $\mathbb{R}^3$, or the open unit ball $\mathbb{B}^3$ in $\mathbb{R}^3$; see \cite[Theorem 3.6.1]{Bea83}. Indeed, if we combine 
\begin{equation*}
{\rm Lip}_\chi(g) = \exp d_{\mathbb{H}^3}(j,g(j)), \quad j = (0,0,1),
\end{equation*}
where $d_{\mathbb{H}^3}$ is the hyperbolic distance in $\mathbb{H}^3$, and
\begin{equation*}
\|A\|^2 = \cosh d_{\mathbb{H}^3}(j,g_A(j)), \quad A\in \SL(2,\mathbb{C}),
\end{equation*}
we obtain ${\rm Lip}_\chi(g_A) = \mu(A)$ (see also \cite[Theorem 2.3.2, p.\,33]{Bea91}).
This approach through hyperbolic geometry is valid in higher dimensions although in these higher dimensions, several of the convenient expressions for M\"obius maps are not available. In the case of four dimensions, these ideas can be expressed in terms of the algebra of quaternions; see, for example, \cite[Theorem 7]{Wil93}.


\end{document}